\documentclass{amsart}
\usepackage{latexsym,amscd,amssymb,url}
\usepackage{extarrows}
\usepackage{verbatim}
\usepackage{dsfont}
\usepackage{tikz-cd}
\usepackage[all,cmtip]{xy}  
\usepackage{graphicx}
\usepackage{xcolor}
\usepackage{enumitem} 
\definecolor{dblue}{rgb}{0,0,.6}
\usepackage[colorlinks=true, linkcolor=dblue, citecolor=dblue, filecolor = dblue, menucolor = dblue, urlcolor = dblue]{hyperref}

\numberwithin{equation}{section}

\newtheorem{theorem}{Theorem}[section]

\theoremstyle{plain}

\newtheorem{step}{Step}

\newtheorem{corollary}[theorem]{Corollary}

\newtheorem{lemma}[theorem]{Lemma}

\newtheorem{proposition}[theorem]{Proposition}
\newtheorem{remark}[theorem]{Remark}

\newcommand{\del}{\partial}
\newcommand{\Z}{\mathbb Z}
\newcommand{\Q}{\mathbb Q}

\newcommand{\C}{\mathbb C}

\newcommand{\RR}{\operatorname{R}}

\newcommand{\CP}{\mathbb P}

\newcommand{\im}{\operatorname{im}}

\newcommand{\Spec}{\operatorname{Spec}}

\newcommand{\Gal}{\operatorname{Gal}}

\newcommand{\pr}{\operatorname{pr}}

\newcommand{\CH}{\operatorname{CH}}

\newcommand{\cl}{\operatorname{cl}}

\newcommand{\Char}{\operatorname{char}}

  \newcommand{\coker}{\operatorname{coker}}

\newcommand{\tors}{\operatorname{tors}}

\newcommand{\et}{\text{\'et}}
\newcommand{\proet}{\text{pro\'et}}

\newcommand{\colim}{\operatorname{colim}}

\newcommand{\ord}{\operatorname{ord}}

\newcommand{\dashedlongrightarrow}{\xymatrix@1@=15pt{\ar@{-->}[r]&}}
\renewcommand{\longrightarrow}{\xymatrix@1@=15pt{\ar[r]&}}
\renewcommand{\mapsto}{\xymatrix@1@=15pt{\ar@{|->}[r]&}}
\renewcommand{\twoheadrightarrow}{\xymatrix@1@=15pt{\ar@{->>}[r]&}}
\newcommand{\hooklongrightarrow}{\xymatrix@1@=15pt{\ar@{^(->}[r]&}}
\newcommand{\congpf}{\xymatrix@1@=15pt{\ar[r]^-\sim&}}
\renewcommand{\cong}{\simeq}

\begin{document}    

\title[On Bloch's map for torsion cycles over non-closed fields]{On Bloch's map for torsion cycles over non-closed fields}

\author{Theodosis Alexandrou} 
\address{Institute of Algebraic Geometry, Leibniz University Hannover, Welfengarten 1, 30167 Hannover , Germany.}
\email{alexandrou@math.uni-hannover.de} 

\author{Stefan Schreieder} 
\address{Institute of Algebraic Geometry, Leibniz University Hannover, Welfengarten 1, 30167 Hannover , Germany.}
\email{schreieder@math.uni-hannover.de}

\date{March 29, 2023} 
%


 \begin{abstract}  
 We generalize Bloch's map on torsion cycles from algebraically closed fields to arbitrary fields.
While Bloch's map over algebraically closed fields is injective for zero-cycles and for cycles of codimension at most two, we show that the generalization to arbitrary fields is only injective for cycles of codimension at most two but in general not for zero-cycles.
Our result implies that Jannsen's cycle class map in integral $\ell$-adic continuous \'etale cohomology is in general not injective on torsion zero-cycles over finitely generated fields.
This answers a question of Scavia and Suzuki.
\end{abstract}

\maketitle

\section{Introduction}
Let $X$ be a smooth variety over a field $k$ and let $\ell$ be a prime invertible in $k$.
In this paper we study the $\ell$-power torsion subgroup $\CH^i(X)[\ell^\infty]$ of the Chow group $\CH^i(X)$ of codimension-$i$ cycles on $X$. 
Following some constructions in \cite{bloch-compositio,Sch-refined}, we define a  cycle map
\begin{align} \label{def:lambda}
\lambda_X^i:\CH^i(X)[\ell^\infty] \longrightarrow \frac{ H^{2i-1}(X_{\et},\Q_\ell/\Z_\ell(i))}{M^{2i-1}(X)} ,
\end{align} 
where $M^{2i-1}(X)\subset H^{2i-1}(X_{\et},\Q_\ell/\Z_\ell(i))$ is defined as follows:
we pick a finitely generated subfield $k_0\subset k$  such that there is a $k_0$-variety $X_0$ with $X=X_0\times_{k_0}k$ and let
$$
M^{2i-1}(X)=\im\left( \lim_{\substack{\longrightarrow \\ k'/k_0}} N^{i-1} H^{2i-1}_{cont}((X_0\times_{k_0}k')_{\et},\Q_\ell(i))\longrightarrow  H^{2i-1}(X_{\et},\Q_\ell/\Z_\ell(i)) \right),
$$
where $k'$ runs through all finitely generated subfields of $k$ that contain $k_0$ and $N^\ast$ denotes the coniveau filtration (cf.\ (\ref{def:N^j}) below).

If $k$ is algebraically closed and $X$ is projective, then $M^{2i-1}(X)=0$ by weight reasons (see Lemma \ref{lem:M^2i-1} below) and the above map coincides with Bloch's map \cite{bloch-compositio}; see Lemma \ref{lem:lambda_X^i=Bloch} below.
If moreover $k=\C$, the map agrees on the subgroup of homologically trivial $\ell$-power torsion cycles with Griffiths' Abel--Jacobi map \cite{griffiths}; see \cite[Proposition 3.7]{bloch-compositio}.

Bloch's map over algebraically closed fields is injective on torsion cycles of codimension $\leq 2$; the non-trivial case $i=2$ is a theorem of Bloch and Merkurjev--Suslin; see \cite[\S 18]{MS}. 
Moreover, for any ground field $k$, Colliot-Th\'el\`ene--Sansuc--Soul\'e showed in \cite[Corollaire 3]{CTSS} that $\CH^2(X)[\ell^\infty]$ is isomorphic to a subquotient of $H^3(X_\et,\Q_\ell/\Z_\ell(2))$. 
The following result generalizes the injectivity theorem of Bloch and Merkurjev--Suslin over algebraically closed fields by making the aforementioned result in \cite{CTSS} more precise; the argument relies on some results from \cite{Sch-refined}. 

\begin{theorem} \label{thm:Sch}
Let $X$ be a smooth variety over a field $k$ and let $\ell$ be a prime invertible in $k$. 
Then $\lambda_X^i$ from (\ref{def:lambda}) is injective 
 for $i=1,2$ 
and induces the following isomorphisms 
$$
\CH^1(X)[\ell^\infty] \cong    H^{1} (X_{\et},\Q_\ell/\Z_\ell(1))/M^1 (X) \ \ \text{and}\ \ \CH^2(X)[\ell^\infty] \cong    N^1H^{3} (X_{\et},\Q_\ell/\Z_\ell(2))/M^{3}(X).
$$
\end{theorem}

In the body of this paper, we prove a version of the above theorem that works for arbitrary algebraic schemes; see Theorem \ref{thm:lambda-body} and Corollary \ref{cor:lambda-body}; to this end one has to replace in the above discussion ordinary cohomology by Borel--Moore cohomology; cf.\ \cite[\S 4]{Sch-refined}. 

Using some Lefschetz hyperplane argument, Bloch shows that the injectivity of $\lambda_X^2$ implies the injectivity of $\lambda_X^{\dim X}$ over algebraically closed fields (see \cite{bloch-compositio}), hence Roitman's theorem \cite{roitman} away from the characteristic.
Similarly, $\lambda_X^{\dim X}$ is known to be injective over finite fields; see e.g.\ \cite{KS,CTSS}.
On the other hand, even over algebraically closed fields, the map $\lambda_X^i$ is for $2<i<\dim X $ in general not injective; see e.g.\ \cite{schoen-product,totaro-JAMS,SV,RS,totaro-annals,Sch-griffiths,scavia-suzuki}.

Roitman's theorem on the injectivity of $\lambda_X^{\dim X}$ over algebraically closed fields turned out to be a very robust statement, with generalizations to  non-proper and even singular spaces; see e.g.\ \cite{levine-AJM,krishna-srinivas,geisser}.
As mentioned above, there is also a generalization to finite fields \cite{KS,CTSS}.
In light of 
the injectivity of $\lambda_X^2$ over arbitrary fields, it is natural to wonder if Roitman's theorem admits also a generalization to arbitrary fields, i.e.\ is $\lambda_{X}^{\dim X}$ always injective?
The main result of this paper answers this question negatively. 

\begin{theorem} \label{thm:main}
Let $\ell$ be a prime and let $k$ be a field of characteristic different from $\ell$.
Then there is a finitely generated field extension $K/k$ and a smooth projective threefold $X$ over $K$ such that $\lambda_X^3$ from (\ref{def:lambda}) is not injective. 
\end{theorem}

The proof of the above result will rely on an adaptation of Schoen's argument in \cite{schoen-product} together with some results on the integral Hodge and Tate conjecture due to Koll\'ar, Hassett--Tschinkel, and Totaro. 
Taking products with projective spaces, we get the following.

\begin{corollary} \label{cor:lambda_X^i-injective}
Let $i\leq n$ be positive integers, let $k$ be a field and let $\ell$ be a prime invertible in $k$.
Then the map $\lambda_X^i$ from (\ref{def:lambda}) is injective for all smooth projective varieties $X$ of dimension $n$ 
 over all finitely generated field extension of $k$ if and only if $i\leq 2$. 
\end{corollary}

We will see in Lemma \ref{lem:lambda-cl} below that over any field, the restriction of Jannsen's cycle class map in continuous \'etale cohomology \cite{jannsen} to $\ell$-power torsion cycles,
 \begin{align} \label{def:cl-jannsen}
\cl^i_X:\CH^i(X)[\ell^\infty] \longrightarrow H^{2i}_{cont}(X_{\et}, \Z_\ell(i)) ,
\end{align}
 factors through $\lambda_X^i$.
In particular, (\ref{def:lambda}) refines Jannsen's cycle class map on torsion cycles. 
As a consequence of Theorem \ref{thm:main}, we immediately get the following, which answers \cite[Question 1.7(a)]{scavia-suzuki}.

\begin{corollary} \label{cor:integral-Jannsen}
For any $n\geq 3$ and any prime $\ell$, there is a smooth projective $n$-fold $X$ over a finitely generated field of characteristic different from $\ell$ such that Jannsen's cycle class map
$$
\cl^n_X:\CH^n(X)[\ell^\infty] \longrightarrow H^{2n}_{cont}(X_{\et}, \Z_\ell(n)) 
$$
is not injective on the subgroup of $\ell$-power torsion classes.
\end{corollary}

Jannsen showed that in the case of divisors, i.e.\ codimension one cycles, his cycle class map for smooth projective varieties over finitely generated fields is injective integrally; see \cite[Remark 6.15(a)]{jannsen} and \cite[(P7.2) in Proposition 6.6]{Sch-refined}.
He further conjectured that after tensoring with $\Q$, his cycle class map is injective for cycles of arbitrary codimension over finitely generated fields.
Scavia and Suzuki exhibited in \cite{scavia-suzuki} several interesting examples that show that tensoring with $\Q$ is really necessary, i.e.\ the integral Jannsen conjecture fails in general. 
The main case left open by \cite{scavia-suzuki} is that of zero-cycles, that we deal with in this paper.
The case of zero-cycles is of particular importance, because injectivity of $\lambda_{X}^{\dim X}$ on $\ell$-power torsion zero-cycles for all smooth projective varieties $X$ over finitely generated fields  would imply the Rost nilpotence conjecture in characteristic zero by \cite{diaz} and the arguments in \cite[\S 9.4]{Sch-refined} and \cite[\S B.2]{Sch-moving} (see also \cite{rosenschon-sawant}).

For the aforementioned application to the Rost nilpotence conjecture it would be enough to prove injectivity of $\lambda_{X}^{\dim X}$ for varieties with a $k$-rational point (essentially because the argument in \cite{diaz} proceeds via base change to the function field, where a rational point is automatic).
The examples in Theorem \ref{thm:main} have no rational point.
However, a variant of our construction gives (at least for the prime $\ell=2$) examples with rational points as well, which puts an end to the hope that the Rost nilpotence conjecture may be a consequence of the injectivity of $\lambda^{\dim X}$ on smooth projective varieties with a rational point.

\begin{theorem}\label{thm:main-2}
There is a finitely generated field $k$ of characteristic zero and a smooth projective threefold $X$ over $k$ such that $X$ has a $k$-rational point and $\lambda_X^3$ from (\ref{def:lambda}) is not injective for the prime $\ell=2$.
\end{theorem}

The examples in Theorem \ref{thm:main} have Kodaira dimension 2,  
while those in Theorem \ref{thm:main-2} have negative Kodaira dimension, as they are conic bundles over surfaces of positive Kodaira dimension.
The question whether $\lambda_X^{\dim X}$ is injective for smooth projective varieties $X$ that are geometrically rationally connected remains open.
 (This question is interesting for the base change $X_{k(X)}$ of any rationally connected variety $X$ over an algebraically closed field $k$ which does not admit an integral decomposition of the diagonal, for examples of the latter see e.g.\ the survey \cite{Sch-survey} and the references therein.) 

\section{Preliminaries}

We recall some of the notation and conventions from \cite{Sch-refined} that we will use.
An algebraic scheme is a separated scheme of finite type over a field.
A variety is an integral algebraic scheme.
A field is finitely generated if it is finitely generated over its prime field.
The $n$-torsion subgroup of an abelian group $G$ is denoted by $G[n]$; the subgroup of elements annihilated by some power of $n$ is denoted by $G[n^\infty]$.
If $\varphi:H\to G$ is a morphism of abelian groups, we denote by slight abuse of notation $G/H:=\coker(\varphi)$.

We fix a field $k$ and a prime $\ell$ invertible in $k$.
For an algebraic scheme $X$ of dimension $d_X$ over $k$, we denote by $\CH^i(X):=\CH_{d_X-i}(X)$ the Chow group of cycles of dimension $d_X-i$ on $X$.
For $A\in \{\Z/\ell^r,\Z_\ell,\Q_\ell,\Q_\ell/\Z_\ell\}$, we let
$$
H^i(X,A(n)):=H^i_{BM}(X,A(n))
$$
be twisted Borel--Moore pro-\'etale cohomology; see \cite[(6.13)-(6.15)]{Sch-refined} and \cite[Proposition 6.6]{Sch-refined}.
%
Some of the most important properties of this functor are collected in \cite[Section 4]{Sch-refined}.
If $X$ is smooth and equi-dimensional, then we have  canonical identifications
\begin{align}\label{eq:ordinary-cohomology}
H^i(X,\Z/\ell^r(n))\cong H^i(X_{\et},\mu_{\ell^r}^{\otimes n}),\ \ H^i(X,\Q_\ell/\Z_\ell(n))=\lim_{\substack{\longrightarrow \\ r}} H^i(X_{\et},\mu_{\ell^r}^{\otimes n}),
\end{align}
$$
H^i(X,\Z_\ell(n))\cong H^i_{cont}(X_{\et},\Z_{\ell}(n)),\ \ \text{and}\ \ H^i(X,\Q_\ell(n))\cong H^i_{cont}(X_{\et},\Q_{\ell}(n))=H^i_{cont}(X_{\et},\Z_{\ell}(n))\otimes_{\Z_\ell} \Q_\ell;
$$
see \cite[Lemma 6.5]{Sch-refined}.

For $j\geq 0$, we denote by $F_jX$ the pro-scheme given by the inverse limit of all open subsets $U\subset X$ with $\dim (X\setminus U)<\dim X-j$.
If $U\hookrightarrow X$ is an open immersion with $\dim U=\dim X$, then there are restriction maps $H^i(X,A(n))\to H^i(U,A(n))$.
As in \cite[Section 5]{Sch-refined}, we define 
$$
H^i(F_jX,A(n)):=\lim_{\substack{\longrightarrow\\ U\subset X}}H^i(U,A(n)) ,
$$ 
where $U$ runs through all open subsets of $X$ that make up the pro-scheme $F_jX$ above.
The coniveau filtration $N^\ast$ on $H^i(X,A(n))$ is then given by
\begin{align} \label{def:N^j}
N^jH^i(X,A(n)):=\ker(H^i(X,A(n))\to H^i(F_{j-1}X,A(n))) ;
\end{align}
see \cite[(5.1)]{Sch-refined}.

For $m\geq j$, there are natural restriction maps
$
H^i(F_mX,A(n))\to H^i(F_jX,A(n))
$
and we denote the image of this map by $F^m H^i(F_jX,A(n))$; see \cite[Definition 5.3]{Sch-refined}.
For a scheme point $x\in X$, we let $H^i(x,A(n)):=H^i(F_0\overline{\{ x\}},A(n))$ where $\overline{\{ x\}}\subset X$ denotes the closure of $x$; note that  $H^0(x,A(0))=A\cdot [x]$, where $[x]\in H^0(x,A(0))$ denotes the fundamental class of $x$; cf.\ \cite[(P3) in Definition 4.2 and Proposition 6.6]{Sch-refined}.

The Gysin sequence induces the following important long exact sequence (see \cite[Lemma 5.8]{Sch-refined})
\begin{align} \label{eq:les}
 \longrightarrow
   H^i(F_jX,A(n))\longrightarrow   H^i(F_{j-1}X,A(n))  \stackrel{\del}\longrightarrow \bigoplus_{x\in X^{(j)}}H^{i+1-2j}(x,A(n-j)) \stackrel{\iota_\ast} \longrightarrow    H^{i+1}(F_jX,A(n)) .
\end{align} 
Since $H^i(x,A(n))=0$ for $i<0$, we deduce $H^i(X,A(n))=H^i(F_jX,A(n))$ for $j\geq \lceil i/2\rceil$; see \cite[Corollary 5.10]{Sch-refined}. 

 Let $\CH^i(X)_{\Z_\ell}:=\CH^i(X)\otimes_{\Z} \Z_\ell$. 
Then there is a cycle class map
$$
\cl^i_X:\CH^i(X)_{\Z_\ell}\longrightarrow H^{2i}(X,\Z_\ell(i));
$$  
see \cite[(7.1)]{Sch-refined}.
This map is induced by the following pushforward map that appears in the above long exact sequence (\ref{eq:les}):
$$
 \bigoplus_{x\in X^{(i)}}H^{0}(x,\Z_\ell(0))= \bigoplus_{x\in X^{(i)}} \Z_\ell[x] \stackrel{\iota_\ast} \longrightarrow    H^{2i}(F_iX,\Z_\ell(i))=H^{2i}(X,\Z_\ell(i)).
$$
If $X$ is smooth and equi-dimensional, then the above cycle class map agrees with Jannsen's cycle class map in continuous \'etale cohomology from \cite{jannsen}; see \cite[Lemma 9.1]{Sch-refined}.

There is a natural coniveau filtration $N^\ast$ on $\CH^i(X)_{\Z_\ell}$, given by the condition that a cycle $[z]$ lies in $N^j$ if and only if  
there is a closed subset $Z\subset X$ of codimension $j$ such that $z$ is rationally equivalent to a homologously trivial cycle on $Z$, i.e.\  
$$
[z]\in \im(\ker(\cl_Z^{i-j})\longrightarrow \CH^i(X)_{\Z_\ell} ) ;
$$
see \cite[Definition 7.3]{Sch-refined}.
In view of \cite[Definition 7.2 and Lemma 7.4]{Sch-refined}, we define
$$
A^i(X)_{\Z_\ell}:=\CH^i(X)_{\Z_\ell}/N^{i-1}\CH^i(X)_{\Z_\ell} \ \ \text{and}\ \ A^i(X)[\ell^\infty]:=A^i(X)_{\Z_\ell}[\ell^\infty].
$$
If the field $k$ is finitely generated, then $A^i(X)_{\Z_\ell}=\CH^i(X)_{\Z_\ell}$ by \cite[Lemma 7.5]{Sch-refined}.

\begin{remark} \label{rem:BM-coho}
Let $X$ be an algebraic $k$-scheme of dimension $d_X$.
Assume that there is a closed embedding $\iota:X\hookrightarrow Y$ into a smooth equi-dimensional algebraic $k$-scheme $Y$ of dimension $d_Y$.
Then the Borel--Moore cohomology groups
$ 
H^i(X,A(n)) =H_{BM}^i(X,A(n))
$ 
above may be identified with ordinary cohomology groups with support as follows:
$$
H_{BM}^i(X,A(n))=H^{i+2c}_X(Y_\proet, \widehat A(n+c)) ,
$$
where $c=d_Y-d_X$; see also \cite[Appendix A]{Sch-moving}.
We explain the above identification in the case where $A=\Z_\ell$; the general case is similar.
To this end, let $\pi_X:X\to \Spec k$ be the structure map and note that by definition in \cite[(6.13)-(6.15)]{Sch-refined} we have 
$$
H_{BM}^i(X,\Z_\ell(n))=\RR^{i-2d_X}\Gamma(X_\proet, \pi_X^! \widehat \Z_\ell(n-d_X)).
$$ 
If $\pi_Y:Y\to \Spec k$ denotes the structure morphism, then $\pi_X=\pi_Y\circ \iota$ and so
$$
\pi_X^!=(\pi_Y\circ \iota)^!=\iota^!\circ \pi_Y^!=\iota^!\circ \pi_Y^\ast(d_Y)[2d_Y],
$$
where we used that $\pi_Y^!=\pi_Y^\ast(d_Y)[2d_Y]$ by Poincar\'e duality; see e.g.\ \cite[Lemma 6.1(4)]{Sch-refined}.
Hence,
$$
H^i_{BM}(X,\Z_\ell(n))= \RR^{i-2d_X+2d_Y}\Gamma(Y_\proet, \iota^! \widehat \Z_\ell(n-d_X+d_Y))=H^{i+2c}_X(Y_\proet,\Z_\ell(n+c)),
$$
where $c:=d_Y-d_X$, and where the above right hand side denotes ordinary pro-\'etale cohomology with support; see \cite[(A.6)]{Sch-moving} and \cite{BS}.
\end{remark}

\section{An auxiliary cycle map}

\subsection{Construction of  \texorpdfstring{$\tilde \lambda_{tr}^i$}{tilde lambda}} \label{subsec:constr-lambda-tilde}
In this subsection we construct a map $\tilde\lambda_{tr}^i$ which is closely related to the transcendental Abel--Jacobi map in \cite{Sch-refined} and which will be used in Section \ref{subsec:constr-lambda} below to construct $\lambda_X^i$ from (\ref{def:lambda}).

To begin with we will need the following simple lemma.

\begin{lemma} \label{lem:N^i-1}
Let $X$ be an algebraic scheme over a field $k$ and let $A\in \{\Z/\ell^r,\Z_\ell,\Q_\ell,\Q_\ell/\Z_\ell\}$ where $\ell$ is a prime invertible in $k$.
Then the natural restriction map $H^{2i-1}(X,A(i))\to H^{2i-1}(F_{i-1}X,A(i)) $ is injective.
Using this to identify each $N^jH^{2i-1}(X,A(i))$ with a subgroup of  $H^{2i-1}(F_{i-1}X,A(i))$, we get a natural isomorphism 
\begin{align}\label{eq:N}
\iota_\ast \left( \ker\left(\del\circ \iota_\ast:\bigoplus_{x\in X^{(i-1)}}H^1( x,A(1)) \to  \bigoplus_{x\in X^{(i)}}[x]A \right) \right)\cong  N^{i-1}H^{2i-1}(X,A(i))   .
\end{align}
\end{lemma}
\begin{proof}
By \eqref{eq:les}, $H^{2i-1}(F_{j}X,A(i))\to H^{2i-1}(F_{j-1}X,A(i)) $ is injective for all $j\geq i$, as $H^{l}(x,A(n))=0$ for $l<0$ and $x\in X$.
Since $F_{j}X=X$ for $j>\dim X$, we deduce that
$H^{2i-1}(X,A(i))\to H^{2i-1}(F_{i-1}X,A(i)) $ is injective, which proves the injectivity claim in the lemma.

Note next that the following sequence is exact by \eqref{eq:les}:
$$
\bigoplus_{x\in X^{(i-1)}} H^{1}(x,A(1))\overset{\iota_{\ast}}{\longrightarrow} H^{2i-1}(F_{i-1}X,A(i))\longrightarrow H^{2i-1}(F_{i-2}X,A(i)) .
$$ 
By the compatibility of the Gysin long exact sequence with proper pushforwards (see \cite[(P2)]{Sch-refined}), we find that 
$$
\iota_\ast \left( \ker\left(\del\circ \iota_\ast:\bigoplus_{x\in X^{(i-1)}}H^1( x,A(1)) \to  \bigoplus_{x\in X^{(i)}}[x]A \right) \right)\subset H^{2i-1}(F_{i-1}X,A(i))
$$
agrees with the image of the map $N^{i-1}H^{2i-1}(X,A(i))\to H^{2i-1}(F_{i-1}X,A(i))$.
This concludes the proof of the lemma.
\end{proof}

Let $X$ be an algebraic $k$-scheme and let $\ell$ be a prime invertible in $k$. By \cite[Lemma 8.1]{Sch-refined}, there is a canonical isomorphism \begin{align} \label{def:psi}
\psi_r:
A^i(X)[\ell^r] \stackrel{\cong}\longrightarrow  \frac{\ker\left(\del\circ \iota_\ast:\bigoplus_{x\in X^{(i-1)}} H^1(x,\mu_{\ell^r}^{\otimes 1})  \longrightarrow  \bigoplus_{x\in X^{(i)}} [x]\Z/\ell^r \right) }{\ker\left(\del\circ \iota_\ast:\bigoplus_{x\in X^{(i-1)}}H^1( x,\Z_{\ell}(1)) \longrightarrow  \bigoplus_{x\in X^{(i)}}[x]\Z_\ell \right)} ,
\end{align}
which we describe in what follows explicitly.
To this end, note that an $\ell^r$-torsion class in $A^i(X)_{\Z_\ell}$ corresponds to a cycle 
$[z]\in \CH^i(X)_{\Z_\ell}$ such that $\ell^r[z]\in N^{i-1}\CH^i(X)_{\Z_\ell}$.
By \cite[Definition 7.2 and Lemma 7.4]{Sch-refined}, this means that there is a class
$$
\xi\in \bigoplus_{x\in X^{(i-1)}}H^1( x,\Z_{\ell}(1)) 
$$
with $\del (\iota_{\ast}\xi)=\ell^r \cdot z$.
It follows that the reduction $\bar \xi$  of $\xi$ modulo $\ell^r$ has trivial residues and we have $\psi_r([z])=[\bar{\xi}]$.

Since $\xi$ has trivial residues modulo $\ell^r$, there is a class $\gamma\in H^{2i-1}(X,\mu_{\ell^r}^{\otimes i})$ with
$$
\gamma=\iota_\ast \bar \xi \in F^i H^{2i-1}(F_{i-1}X,\mu_{\ell^r}^{\otimes i}).
$$
Note that $\gamma$ is uniquely determined by $\bar \xi$, because the map
$$
H^{2i-1}( X,\mu_{\ell^r}^{\otimes i})\longrightarrow H^{2i-1}(F_{i-1}X,\mu_{\ell^r}^{\otimes i})
$$
is injective by Lemma \ref{lem:N^i-1}. 
Using this we will in what follows implicitly identify $ \iota_\ast \bar \xi $ and $\gamma$ with each other.

The image of $\gamma= \iota_\ast \bar \xi $ in $H^{2i-1}(X,\Q_\ell/\Z_\ell(i))$ is well-defined (i.e.\ depends only on the class $[z]$) up to elements of the form $\iota_\ast \bar \zeta$ for 
\begin{align}\label{eq:zeta}
\zeta\in \ker\left(\del\circ \iota_\ast:\bigoplus_{x\in X^{(i-1)}}H^1( x,\Z_{\ell}(1)) \longrightarrow  \bigoplus_{x\in X^{(i)}}[x]\Z_\ell \right).
\end{align}
By Lemma \ref{lem:N^i-1}, the subgroup of $H^{2i-1}(X,\Q_\ell/\Z_\ell(i))$ generated by classes of the form $\iota_\ast \bar \zeta$ with $\zeta$ as in (\ref{eq:zeta}) agrees with the image of $N^{i-1}H^{2i-1}(X,\Q_\ell(i))$ and so we find that
the class
$$
\tilde \lambda_{tr}^i([z]):=-[ \iota_\ast \bar \xi ]\in H^{2i-1}(X,\Q_\ell/\Z_\ell(i))/N^{i-1}H^{2i-1}(X,\Q_\ell(i))
$$
is well-defined, giving rise to a map
\begin{align} \label{def:tilde-lambda}
\tilde \lambda_{tr}^i:A^i(X)[\ell^\infty]\longrightarrow H^{2i-1}(X,\Q_\ell/\Z_\ell(i))/N^{i-1}H^{2i-1}(X,\Q_\ell(i)).
\end{align}
(The minus sign is necessary to make our map compatible with Bloch's map; cf.\ \cite[p.\ 112]{bloch-compositio} and \cite[Proposition 8.3]{Sch-refined}.)

\subsection{Basic properties of \texorpdfstring{$\tilde \lambda_{tr}^i$}{tilde lambda}}

\begin{lemma} \label{lem:tilde-lambda-vslambda} 
The restriction of $ \tilde \lambda_{tr}^i$ to the subgroup of classes $[z]\in A^i(X)[\ell^\infty]$ with $\cl_X^i(z)=0$ coincides with the transcendental Abel--Jacobi map $\lambda_{tr}^i$ from \cite[\S 7.5]{Sch-refined}. 
\end{lemma}
\begin{proof}
Assume that $\cl_X^i(z)=0$. Following the construction in \cite[\S 7.5]{Sch-refined}, there is a 
 class $\alpha\in H^{2i-1}(F_{i-1}X,\Z_\ell(i))$ with $\del \alpha=z$ and a class $\beta\in H^{2i-1}(X,\Z_\ell(i))$ with
\begin{align} \label{eq:beta=l^r-alpha}
\beta=\ell^r\alpha-\iota_\ast \xi\in F^i H^{2i-1}(F_{i-1}X,\Z_\ell(i))  
\end{align}
 for some $\xi\in \bigoplus_{x\in X^{(i-1)}}H^1(x,\Z_\ell(1))$. 
 We then think about $\beta/\ell^r$ as class in $H^{2i-1}(X,\Q_\ell(i))$ and project this further to a class $[\beta/\ell^r]\in H^{2i-1}(X,\Q_\ell/\Z_\ell(i))$.
 By definition in \textit{loc.\ cit.}, this class represents $\lambda_{tr}^i([z])$:
 $$
\lambda_{tr}^i([z])=[\beta/\ell^r] \in \frac{H^{2i-1}(X,\Q_\ell/\Z_\ell(i))}{N^{i-1}H^{2i-1}(X,\Q_\ell(i) )} .
 $$
 We aim to see that this coincides with $\tilde \lambda_{tr}^i([z])$ defined above.
 To this end, note that (\ref{eq:beta=l^r-alpha}) implies $\del \iota_\ast \xi=\ell^r\del \alpha=\ell^r\cdot z$ and so
 $$
 \tilde \lambda_{tr}^i([z])=-[\iota_\ast \bar \xi],
 $$
 where $\bar \xi$ is the reduction modulo $\ell^r$ of $\xi\in \bigoplus_{x\in X^{(i-1)}}H^1(x,\Z_\ell(1))$.
By (\ref{eq:beta=l^r-alpha}), we have $\bar \beta=-\iota_\ast \bar\xi$ and so the claim in the lemma reduces to the simple observation that the following diagram is commutative:
$$
\xymatrix{
H^{2i-1}(F_{i-1}X,\Z_\ell(i))\ar[r]^{\cdot (1/\ell^r)} \ar[d]^{\mod \ell^r}& H^{2i-1}(F_{i-1}X,\Q_\ell(i))\ar[r] & H^{2i-1}(F_{i-1}X,\Q_\ell/\Z_\ell(i))\ar[d]^{=}\\
H^{2i-1}(F_{i-1}X,\mu_{\ell^r}^{\otimes i})\ar[rr]&& \colim_s H^{2i-1}(F_{i-1}X,\mu_{\ell^s}^{\otimes i})
} .
$$
This concludes the proof of the lemma.
\end{proof}

It is shown in \cite{Sch-refined} that $\lambda_{tr}^i$ is injective for $i\leq 2$; cf.\ \cite[Theorem 9.4 and Corollary 9.5]{Sch-refined} or \cite[Theorem 1.8(2)]{Sch-refined}.
The following result extends this to $\tilde \lambda_{tr}^i$ as follows.

\begin{proposition}\label{prop:tilde-lambda}
Let $X$ be an algebraic scheme over a finitely generated field $k$ and let $\ell$ be a prime invertible in $k$.
Then 
$$
\tilde \lambda_{tr}^i:A^i(X)[\ell^\infty]\longrightarrow H^{2i-1}(X,\Q_\ell/\Z_\ell(i))/N^{i-1}H^{2i-1}(X,\Q_\ell(i))
$$
is injective for $i\leq 2$.
\end{proposition}
\begin{proof} 
Let $[z_{0}]\in\ker(\tilde \lambda_{tr}^i)$.
We first show that $[z_0]$ lies in the kernel of the cycle class map. 
By (\ref{def:psi}), we find that there is a class $\xi\in\bigoplus_{x\in X^{(i-1)}}H^{1}(x,\Z_{\ell}(1))$ such that $\ell^{r}z_{0}=\del(\iota_{\ast}\xi)$ and such that $\tilde \lambda_{tr}^i([z_{0}])=-[\iota_{\ast}\overline{\xi}]=0$. 
Especially, we may assume $\iota_{\ast}\overline{\xi}\in\im (N^{i-1}H^{2i-1}(X,\Z_{\ell}(i))\longrightarrow H^{2i-1}(X,\mu^{\otimes i}_{\ell^{r}}))$ and so by \eqref{eq:N} we obtain $\iota_{\ast}\overline{\xi}=\iota_{\ast}\overline{\zeta}$ for some 
$$
\zeta\in\ker \left( \del\circ\iota_{\ast}\colon\bigoplus_{x\in X^{(i-1)}} H^{1}(x,\Z_{\ell}(1))\longrightarrow \bigoplus_{x\in X^{(i)}}[x]\Z_{\ell} \right).
$$
Exactness of the Bockstein sequence then gives $\iota_{\ast}\xi=\iota_{\ast}\zeta+\ell^{r}\alpha$ for some $\alpha\in H^{2i-1}(F_{i-1}X,\Z_{\ell}(i))$.
Since $\del \iota_\ast \zeta=0$, we get that $z_{0}=\del\alpha$ and so $\cl^{i}_{X}([z_{0}])=\iota_\ast[z_0]=0$ by exactness of \eqref{eq:les}.
Hence, $z_0$ has trivial cycle class on $X$, as claimed.
By Lemma \ref{lem:tilde-lambda-vslambda}, it then follows that $z_0$ lies in the kernel of the transcendental Abel--Jacobi map from \cite{Sch-refined}: $\lambda_{tr}^i([z_0])=0$.
In other words, we have proved that $\ker(\tilde \lambda_{tr}^i)=\ker( \lambda_{tr}^i)$ (but note that the two maps are defined on different domains: $\lambda_{tr}^i$ is defined on homologically trivial $\ell^\infty$-torsion cycles, while $\tilde \lambda_{tr}^i$ is defined on arbitrary $\ell^\infty$-torsion cycles).
Since $\ker(\tilde \lambda_{tr}^i)=\ker( \lambda_{tr}^i)$, \cite[Theorem 1.8(2)]{Sch-refined} yields
$$
\ker(\tilde \lambda_{tr}^i)=  H^{2i-2}_{i-3,nr} (X,\Q_\ell /\Z_\ell (i)) /G^{i}H^{2i-2}_{i-3,nr} (X,\Q_\ell/\Z_\ell(i) ) ,
$$
where $H^i_{j,nr}(X,A(n))=\im(H^i(F_{j+1}X,A(n))\to H^i(F_{j}X,A(n)))$ denotes the $j$-th refined unramified cohomology; cf.\ \cite{Sch-refined}.
Since $F_jX=\emptyset$ for $j<0$, we get that $H^i_{j,nr}(X,A(n))=0$ for $j<0$ and so
$ 
\ker(\tilde \lambda_{tr}^i)=0
$ 
for $i<3$.
This proves the proposition.
\end{proof}

The following result is motivated by the description of the image of $\lambda_{tr}^i$ in \cite[Proposition 7.16]{Sch-refined}.

\begin{proposition}\label{prop:image-tilde-lambda} 
Let $X$ be an algebraic $k$-scheme and let $\ell$ be a prime invertible in $k$. Then the map $\tilde{\lambda}^{i}_{tr}$ from \eqref{def:tilde-lambda} satisfies: $$\im(\tilde{\lambda}^{i}_{tr})=\frac{N^{i-1}H^{2i-1}(X,\Q_{\ell}/\Z_{\ell}(i))}{N^{i-1}H^{2i-1}(X,\Q_{\ell}(i))}.$$
\end{proposition}
\begin{proof} 
Recall from Section \ref{subsec:constr-lambda-tilde} that $\tilde\lambda^{i}_{tr}$ was constructed as the direct limit of maps $$\tilde\lambda^{i}_{r}\colon A^{i}(X)[\ell^{r}]\longrightarrow\frac{H^{2i-1}(X,\mu^{\otimes i}_{\ell^{r}})}{N^{i-1}H^{2i-1}(X,\Z_{\ell}(i))}$$ over all positive integers $r\geq 1$. 
As taking direct limits is an exact functor, it suffices to show that 
\begin{align}\label{eq:image-tilde-lambda} 
\im(\tilde{\lambda}^{i}_{r})=\frac{N^{i-1}H^{2i-1}(X,\mu^{\otimes i}_{\ell^{r}})}{N^{i-1}H^{2i-1}(X,\Z_{\ell}(i))}\ \ \ \text{for all}\ r\geq 1.
\end{align} 
Let $[z_{0}]\in A^{i}(X)[\ell^{r}]$. 
Then by \eqref{def:psi} and the construction in Section \ref{subsec:constr-lambda-tilde}, there is a class 
$$
\xi\in\bigoplus_{x\in X^{(i-1)}} H^{1}(x,\Z_{\ell}(1))
$$ 
such that $\del(\iota_{\ast}\xi)=\ell^{r}z_{0}$ and such that $\tilde\lambda^{i}_{r}([z_{0}])=-[\iota_{\ast}\overline{\xi}]$. 
Since the following sequence $$\bigoplus_{x\in X^{(i-1)}} H^{1}(x,\mu^{\otimes 1}_{\ell^{r}})\overset{\iota_{\ast}}{\longrightarrow} H^{2i-1}(F_{i-1}X,\mu^{\otimes i}_{\ell^{r}})\longrightarrow H^{2i-1}(F_{i-2}X,\mu^{\otimes i}_{\ell^{r}})$$ is exact by \eqref{eq:les}, we find $\iota_{\ast}\overline{\xi}\in N^{i-1}H^{2i-1}(X,\mu^{\otimes i}_{\ell^{r}})$. This proves the inclusion “$\subseteq$” for \eqref{eq:image-tilde-lambda}.

On the other hand, if $\beta\in N^{i-1}H^{2i-1}(X,\mu^{\otimes i}_{\ell^{r}})$ then, by \eqref{eq:N}, we can pick 
$$
\overline{\xi}\in\ker \left(\del\circ\iota_{\ast}\colon\bigoplus_{x\in X^{(i-1)}}H^{1}(x,\mu^{\otimes 1}_{\ell^{r}})\longrightarrow \bigoplus_{x\in X^{(i)}}[x]\Z/\ell^{r} \right)
$$ 
such that $\iota_{\ast}\overline{\xi}=\beta$. 
Hilbert 90 implies that $H^1(x,\Z_\ell(1))\to H^{1}(x,\mu^{\otimes i}_{\ell^{r}})$ is surjective for all $x\in X$; see
 \cite[(P6) in Definition 4.4 and Proposition 6.6]{Sch-refined}.
 Hence, there is a class $\xi\in\bigoplus_{x\in X^{(i-1)}} H^{1}(x,\Z_{\ell}(1))$ whose reduction modulo $\ell^{r}$ is $\overline{\xi}\in\bigoplus_{x\in X^{(i-1)}}H^{1}(x,\mu^{\otimes i}_{\ell^{r}})$. 
In particular, $0=\del\beta=\del(\iota_{\ast}\overline{\xi})$ yields $\del(\iota_{\ast}\xi)=\ell^{r}z_{0}$ for some $z_{0}\in\bigoplus_{x\in X^{(i)}}[x]\Z_{\ell}$ and thus $\tilde\lambda^{i}_{r}(-[z_{0}])=[\iota_{\ast}\overline{\xi}]=[\beta]$. 
This finishes the proof of the proposition.
\end{proof}

By Lemma \ref{lem:tilde-lambda-vslambda}, $\tilde \lambda^i_{tr}$ and $\lambda_{tr}^i$ agree on $\ell$-power torsion classes with trivial cycle class.
Moreover, in the proof of Proposition \ref{prop:tilde-lambda} we showed that $\tilde \lambda^i_{tr}$ and $\lambda_{tr}^i$ have the same kernel.
The following lemma shows more generally that $\tilde \lambda^i_{tr}$ factors through the cycle class map.

\begin{lemma} \label{lem:tilde-lambda-vscl}
Let $X$ be an algebraic $k$-scheme which admits a closed embedding into a smooth equi-dimensional algebraic $k$-scheme (e.g.\ $X$ is quasi-projective).
Then the composition of $ \tilde \lambda_{tr}^i$ from (\ref{def:tilde-lambda}) with the negative of the Bockstein map 
$$
\delta:H^{2i-1}(X,\Q_\ell/\Z_\ell(i))\longrightarrow H^{2i}(X,\Z_\ell(i))
$$
agrees with the cycle class map $\cl_X^i:A^i(X)[\ell^\infty]\longrightarrow H^{2i}(X,\Z_\ell(i))$.
\end{lemma}
\begin{proof} 
We follow the proof of \cite[Proposition 1]{CTSS}. 
By topological invariance of the pro-\'etale site, we may up to replacing $k$ by its perfect closure assume that $k$ is perfect (cf.\ \cite[Lemma 5.4.2]{BS}). 

Let $[z_{0}]\in A^{i}(X)[\ell^\infty]$. 
Then there is a closed subset $W\subset X$ of pure codimension $i-1$ and a class $\xi\in H^{1}(F_{0}W,\Z_{\ell}(1))$ such that $\del\xi=\ell^{r}z_{0}$ for some integer $r\geq 0$. 
In particular, $\tilde\lambda^{i}_{tr}([z_{0}])=-[\iota_{\ast}\overline{\xi}],$ where $\overline{\xi}\in H^{1}(W,\mu_{\ell^{r}})$ is the reduction of $\xi$ modulo $\ell^{r}$ and $\iota\colon W\hookrightarrow X$ the obvious closed embedding.
Since the Bockstein map is compatible with proper pushforwards (cf.\ \cite[(P5) in Definition 4.4 and Proposition 6.6]{Sch-refined}), the lemma follows if we can show that
\begin{align} \label{eq:-delta=cl-on-W}
-\delta(\overline{\xi})=\cl_W^1(z_0)\in H^2(W,\Z_\ell(1)).
\end{align}
The above statement does not depend on the ambient space $X$.
Since $X$ can be embedded into a smooth equi-dimensional $k$-scheme by assumption, we may thus from now on assume without loss of generality that $X$ is smooth and equi-dimensional.

In what follows, for a closed subset $Z\subset X$ and $A\in \{\Z/\ell^r,\Z_\ell,\Q_\ell,\Q_\ell/\Z_\ell\}$ the group $H^{i}_{Z}(X,A(n))$ stands for ordinary pro-\'etale cohomology with support; if $c=\dim X-\dim Z$, then
$$
H^{i-2c}_{BM}(Z,A(n))=H^{i}_{Z}(X,A(n))
$$
by Remark \ref{rem:BM-coho}. 
The equation (\ref{eq:-delta=cl-on-W}) then translates into the claim that  
\begin{align} \label{eq:-delta=cl-on-W-2}
\alpha_1:=-\delta(\overline{\xi}) \in H^{2i}_W(X,\Z_\ell(i))\ \ \ \text{and}\ \ \  \alpha_2:=\cl_W^1(z_0) \in H^{2i}_W(X,\Z_\ell(i)) \ \ \ \text{coincide}.
\end{align} 

We aim to describe the class $\alpha_{2}=\cl_W^1(z_0)$. 
Since $\xi\in H^{1}(F_{0}W,\Z_{\ell}(1)),$ we may pick a closed subset $W'\subset W$ of pure codimension one such that $\xi$ admits a lift to $H^{1}(W\setminus W',\Z_{\ell}(1))=H^{2i-1}_{W\setminus W'}(X\setminus W',\Z_{\ell}(i))$. 
Then $\del\xi=\ell^{r}z_{0}\in H^{2i}_{W'}(X,\Z_{\ell}(i))=\bigoplus_{x\in W'^{(0)}}[x]\Z_{\ell}$ and  $\alpha_{2}$ is the image of $z_{0}$ in $H^{2i}_{W}(X,\Z_{\ell}(i))$. 

In order to show $\alpha_1=\alpha_2$, we take a Cartan-Eilenberg injective resolution 
$$0\longrightarrow I^\bullet \longrightarrow J^\bullet \stackrel{p}\longrightarrow K^\bullet \longrightarrow 0$$ 
for the following short exact sequence of sheaves on $X_{\text{pro\'et}}$ 
$$0\longrightarrow \hat{\Z}_{\ell}(i)\overset{\times\ell^{r}}\longrightarrow\hat{\Z}_{\ell}(i)\longrightarrow \mu^{\otimes i}_{\ell^{r}}\longrightarrow 0$$
and we consider the following commutative diagram of complexes of abelian groups.\begin{center}\begin{tikzcd}
            & 1 \arrow[d]                                                       & 1 \arrow[d]                                                      & 1 \arrow[d]                                                                              &   \\
1 \arrow[r] &  H^{0}_{W'}(X,I^\bullet ) \arrow[r] \arrow[d]      &  H^{0}_{W}(X,I^\bullet ) \arrow[r] \arrow[d]      &  H^{0}_{W\setminus W'}(X\setminus W',I^\bullet ) \arrow[r] \arrow[d]      & 1 \\
1 \arrow[r] &  H^{0}_{W'}(X,J^\bullet ) \arrow[r] \arrow[d, "p"] & { H^{0}_{W}(X,J^\bullet )} \arrow[r] \arrow[d, "p"] & { H^{0}_{W\setminus W'}(X\setminus W',J^\bullet )} \arrow[r] \arrow[d, "p"] & 1 \\
1 \arrow[r] & { H^{0}_{W'}(X,K^\bullet )} \arrow[r] \arrow[d]      & {H^{0}_{W}(X,K^\bullet )} \arrow[r] \arrow[d]      & {H^{0}_{W\setminus W'}(X\setminus W',K^\bullet )} \arrow[r] \arrow[d]      & 1 \\
            & 1                                                                 & 1                                                                & 1                                                                                        &  
\end{tikzcd}\end{center} 
As higher cohomology (with support) of injective objects vanishes, we find that both the rows and columns of the above diagram are exact. We shall denote the differential of these complexes by the letter $d$.

Let $\eta\in H^{0}_{W}(X,K^{2i-1})$ be a lift of $\overline{\xi}\in H^{2i-1}_{W}(X,\mu^{\otimes i}_{\ell^{r}})$. 
We claim that there exists a lift $\omega\in H^{0}_{W}(X,J^{2i-1})$ of $\eta$ such that the image $\omega '$ of $\omega$ in $H^{0}_{W\setminus W'}(X\setminus W',J^{2i-1})$ satisfies $d\omega '=0$. 
Indeed, let $\eta '$ denote the image of $\eta$ in $H^{0}_{W\setminus W'}(X\setminus W',K^{2i-1})$. Since $\overline{\xi}$ is the reduction modulo $\ell^{r}$ of a class $\xi\in H^{2i-1}_{W\setminus W'}(X\setminus W',\Z_{\ell}(i))$ and the reduction map $$H^{2i-1}_{W\setminus W'}(X\setminus W',\Z_{\ell}(i))\longrightarrow H^{2i-1}_{W\setminus W'}(X\setminus W',\mu^{\otimes i}_{\ell^{r}})$$ is induced by $$p\colon H^{0}_{W\setminus W'}(X\setminus W',J^\bullet)\longrightarrow H^{0}_{W\setminus W'}(X\setminus W',K^\bullet),$$ we can find a class $\tau'_{1}\in H^{0}_{W\setminus W'}(X\setminus W',J^{2i-1})$ with $d\tau'_{1}=0$ such that $p(\tau'_{1})=\eta'+d(\kappa ')$ for some $\kappa\in H^{0}_{W}(X,K^{2i-2})$ with image $\kappa'\in H^{0}_{W\setminus W'}(X\setminus W',K^{2i-2})$. Replacing $\eta$ with $\eta+d\kappa,$ we may assume $p(\tau'_{1})=\eta '$. Let $\tau_{1}\in H^{0}_{W}(X,J^{2i-1})$ be a lift of $\tau_{1} '$. Exactness of the above diagram yields $\eta-p(\tau_{1})=p(\iota_{\ast}\tau_{2})$ for some $\tau_{2}\in H^{0}_{W'}(X,J^{2i-1})$ and we take $\omega:=\tau_{1}+\iota_\ast \tau_{2}$, where $\iota_\ast \tau_2$ denotes the image of $\tau_2$ in $H^{0}_{W}(X,J^{2i-1})$.
Then $p(\omega)=\eta$ and the image $\omega '$ of $\omega$ in $H^{0}_{W\setminus W'}(X\setminus W',J^{2i-1})$ satisfies $d\omega '=d\tau_1'=0$, as claimed above.

Using that the Bockstein map $\delta\colon H^{2i-1}_{W}(X,\mu^{\otimes i}_{\ell^{r}})\longrightarrow H^{2i}_{W}(X,\Z_{\ell}(i))$ is the coboundary map of the middle vertical short exact sequence in the diagram above, we deduce $\alpha_{1}=[d\omega]$, where we regard $d\omega$ as an element in $H^{0}_{W}(X,I^{2i})$.

Finally we recall the construction of the class $\alpha_{2},$ so as to verify our claim, i.e. $\alpha_{1}=\alpha_{2}$. Indeed, the class $\omega '$ corresponds to a lift $\xi$ of $\overline{\xi}\in H^{2i-1}_{W\setminus W'}(X\setminus W',\mu^{\otimes i}_{\ell^{r}})$ in $ H^{2i-1}_{W\setminus W'}(X\setminus W',\Z_{\ell}(i))$ and an easy diagram chase gives that $\del\xi\in H^{2i}_{W'}(X,\Z_{\ell}(i))$ is the cohomology class of $d\omega\in H^{0}_{W'}(X,J^{2i})$. 
Since $p(d\omega)=d\eta=0$, we deduce $d\omega\in H^{0}_{W'}(X,I^{2i})$ and the corresponding class in $H^{2i}_{W'}(X,\Z_{\ell}(i))$
agrees with $z_{0}$. 
Hence the image of $z_{0}$ in $H^{2i}_{W}(X,\Z_{\ell}(i))$ is the cohomology class of $d\omega \in H^{0}_{W}(X,I^{2i})$, showing $\alpha_{1}=\alpha_{2},$ as we want.
This concludes the proof of the lemma.
\end{proof}

\begin{lemma} \label{lem:lambda-z1xz2}
Let $X_j$ for $j=1,2$ be smooth projective $k$-varieties and let $\ell$ be a prime invertible in $k$.
Let $z_j\in \CH^{i_j}(X_j)$ be classes such that $z_2$ is $\ell^r$-torsion.
Then
$$
\tilde \lambda_{tr}^{i_{1}+i_{2}}(z_1\times z_2)=\pr_1^\ast \left( \overline{ \cl_{X_1}^{i_1}(z_1)}\right)\cup \pr_2^\ast \left(  \tilde \lambda_{tr}^{i_2}([z_2]) \right),
$$
where $\overline{ \cl_{X_1}^{i_1}(z_1)}\in H^{2i_1}(X_1,\Q_\ell/\Z_\ell(i_1))$ denotes the reduction modulo $\ell^r$ of $ \cl_{X_1}^{i_1}(z_1)$, i.e.\ the image of $ \cl_{X_1}^{i_1}(z_1)$ via the composition
$$
H^{2i_1}(X_1,\Z_\ell(i_1))\longrightarrow H^{2i_1}(X_1,\mu_{\ell^r}^{\otimes i_1})\longrightarrow \colim_s  H^{2i_1}(X_1,\mu_{\ell^s}^{\otimes i_1})=H^{2i_1}(X_1,\Q_\ell/\Z_\ell(i_1)) .
$$
\end{lemma}
\begin{proof}
By topological invariance of the \'etale, resp.\ pro-\'etale site, we may up to replacing $k$ by its perfect closure assume that $k$ is perfect (cf.\ \cite[Lemma 5.4.2]{BS}).

We will frequently cite properties of cohomology with support from \cite[Appendix A]{Sch-moving}, which applies to our setting by Remark \ref{rem:BM-coho} above.

There is a closed subset $W_2\subset X$ of pure codimension $i_2-1$ and a class $\xi\in H^1(F_0W_2,\Z_\ell(1))$ such that $\del\xi=\ell^r z_2$.
We then have $\tilde \lambda_{tr}^{i_2}([z_2]) = -[(\iota_2)_\ast \overline{\xi}]$, where $\overline \xi$ denotes the reduction of $\xi$ modulo $\ell^r$ and $(\iota_2)_\ast$ denotes the pushforward induced by the closed embedding $\iota_2:W_2\hookrightarrow X_2$.
If $\iota_1:W_1\hookrightarrow  X_1$ denotes the inclusion of the support of $z_1$, then $c:=\cl_{W_1}^0(z_{1})\in H^0(W_1,\Z_\ell(0))=H^0(F_{0}W_1,\Z_\ell(0))$ satisfies $\cl_{X_1}^{i_1}([z_1])=(\iota_1)_\ast c$.
We may then consider the class
$$
p_1 ^\ast c\cup  p_2^\ast \xi \in H^1(F_0(W_1\times W_2),\Z_\ell(1)),
$$
where $p_i^\ast$ denotes the pullback induced by the projection map $p_i:W_1\times W_2\to W_i$.
Note that we use here the reduction step that $k$ is perfect as it implies that for each $i$, $W_i$ is generically smooth and equi-dimensional in which case Borel--Moore cohomology  agrees with ordinary cohomology.
In particular, the cup product used above exists.

The residue of the above class is given by 
$$
\del \left( p_1^\ast c\cup p_2^\ast \xi \right)= p_1^\ast c\cup  \del \left( p_2^\ast \xi \right)=\ell^r\cdot(z_1\times z_2);
$$
cf.\ \cite[Lemma 2.4]{Sch-survey}.
Hence,
$$
\tilde \lambda_{tr}^{i_{1}+i_{2}}(z_1\times z_2)=-(\iota_1\times \iota_2)_\ast ( \overline{ p_1^\ast c\cup p_2^\ast \xi } ).
$$
In particular, the lemma follows once we have proven the following claim: 
\begin{align} \label{eq:lem:lambda-z1xz2}
   -(\iota_1\times \iota_2)_\ast( \overline{ p_1^\ast c\cup p_2^\ast \xi } )=\pr_1^\ast \left( (\iota_1)_\ast \overline{ c}\right)\cup \pr_2^\ast \left( - (\iota_2)_\ast \overline \xi \right)=
\pr_1^\ast \left( \overline{ \cl_{X_1}^{i_1}(z_1)}\right)\cup \pr_2^\ast \left(  \tilde \lambda_{tr}^{i_2}([z_2]) \right) ,
\end{align}
where $\pr_i^\ast$ denotes the pullback induced by the projection $\pr_i:X_1\times X_2\to X_i$.
The second equality in (\ref{eq:lem:lambda-z1xz2}) is clear and so it suffices to prove the first.
By linearity, it suffices to prove this in the case where $W_1$ is
irreducible and $c=1\cdot [W_1]\in H^0(W_1,\Z_\ell(0))$ is the fundamental class.
It then suffices to prove
\begin{align} \label{eq:lem:lambda-z1xz2-2}
 (\iota_1\times \iota_2)_\ast( \overline{ p_2^\ast \xi } )=\pr_1^\ast \left(\overline{\cl_{X_1}^{i_1}(W_1)}\right)\cup \pr_2^\ast \left((\iota_2)_\ast \overline \xi \right).
\end{align}
We consider the closed embeddings
$$
f:W_1\times W_2\to X_1\times W_2 \ \ \text{and}\ \ g:X_1\times W_2\to X_1\times X_2.
$$
Note that $\iota_1\times \iota_2=g\circ f$.
Let further $q_1: X_1\times W_2 \to X_1$ and $q_2:X_1\times W_2\to W_2$ denote the projections.
Then $p_2=q_2\circ f$ and so
$$
 (\iota_1\times \iota_2)_\ast( \overline{ p_2^\ast \xi } )=g_\ast f_\ast (f^\ast q_2^\ast \overline{\xi}) .
$$
By the projection formula (see e.g.\ \cite[Lemma A.19]{Sch-moving}), applied to $f$, we thus get
$$
 (\iota_1\times \iota_2)_\ast( \overline{ p_2^\ast \xi } )=g_\ast (f_\ast 1\cup q_2^\ast \overline{ \xi}) .
$$ 
Note that in the above equation,
$$
f_\ast 1 =g^\ast \cl_{X_1\times X_2}^{i_1}(W_1\times X_2)\in H^{2i_1}(X_1\times W_2,\mu_{\ell^r}^{\otimes i_1}).
$$
Hence
$$
 (\iota_1\times \iota_2)_\ast( \overline{ p_2^\ast \xi } )=g_\ast (g^\ast \cl_{X_1\times X_2}^{i_1}(W_1\times X_2)\cup q_2^\ast \overline{ \xi}) .
$$
Applying the projection formula with respect to $g$ thus yields
$$
 (\iota_1\times \iota_2)_\ast( \overline{ p_2^\ast \xi } )= \cl_{X_1\times X_2}^{i_1}(W_1\times X_2)\cup g_\ast  q_2^\ast \overline{ \xi} .
$$
To prove (\ref{eq:lem:lambda-z1xz2-2}) it thus suffices to show that
\begin{align} \label{eq:lem:lambda-z1xz2-3}
\cl_{X_1\times X_2}^{i_1}(W_1\times X_2)=\pr_1^\ast \left(\cl_{X_1}^{i_1}(W_1)\right)
\ \ \text{and}\ \ 
g_\ast  q_2^\ast \overline{ \xi}=\pr_2^\ast \left((\iota_2)_\ast\overline\xi\right).
\end{align}
The first identity follows directly by the compatibility of pullbacks in cohomology and Chow groups via the cycle class map (see e.g.\ \cite[Lemma A.21]{Sch-moving}).
The second identity follows from the compatibility of pullbacks and pushforwards as outlined e.g.\ in \cite[Lemma A.12(2)]{Sch-moving}, applied to the commutative diagram
$$
\xymatrix{
X_1\times (W_2\setminus Z_2)\ar[r] \ar[d]& X_1\times (X_2\setminus Z_2)\ar[d] \\
 W_2\setminus Z_2 \ar[r] &  X_2\setminus Z_2 ,
}
$$
where $Z_2\subset W_2$ is a closed subset that is nowhere dense and which contains the singular locus of $W_2$.
This concludes the proof of the lemma.
\end{proof}

\begin{corollary} \label{cor:lambda-z1xz2}
Let $X_j$ for $j=1,2$ be smooth projective $k$-varieties and let $\ell$ be a prime invertible in $k$.
Let $z_j\in \CH^{i_j}(X_j)$ be classes such that $z_2$ is $\ell^r$-torsion for some positive integer $r$.
Assume that $\cl^{i_1}_{X_1}(z_1)$ is zero modulo $\ell^r$.
Then the $\ell^r$-torsion cycle
$$
z:=z_1\times z_2\in \CH^{i_1+i_2}(X_1\times X_2)
$$
lies in the kernel of $\tilde \lambda_{tr}^{i_1+i_2}$.
\end{corollary}
\begin{proof}
This is an immediate consequence of Lemma \ref{lem:lambda-z1xz2}.
\end{proof}

\section{Passing to the limit over finitely generated subfields}

\subsection{Construction of \texorpdfstring{$\lambda_X^i$}{lambda}} \label{subsec:constr-lambda}

Let $X$ be an algebraic scheme over a field $k$ and let $k_0\subset k$ be a finitely generated subfield such that there is a variety $X_0$ over $k_0$ with $X=X_0\times_{k_0} k$.
For any finitely generated subfield $k'\subset k$ with $k_0\subset k'$, we consider $X_{k'}:=X_0\times_{k_0}k'$.
We then have
$$
\CH^i(X)=\lim_{\substack{\longrightarrow\\ k'/k_0}}\CH^i(X_{k'})\ \ \text{and}\ \  \CH^i(X)[\ell^\infty]=\lim_{\substack{\longrightarrow\\ k'/k_0}}A^i(X_{k'})[\ell^\infty],
$$
where the limit runs through all finitely generated subfields $k'\subset k$ with $k_0\subset k'$ and where we used $\CH^i(X_{k'})_{\Z_\ell}=A^i(X_{k'})_{\Z_\ell}$ for $k'$ finitely generated; cf.\ \cite[Lemma 7.5]{Sch-refined}.
Using this, we define $\lambda_X^i$  as the direct limit of $\tilde \lambda_{tr}^i$ (see Section \ref{subsec:constr-lambda-tilde}), applied to $X_{k'}$ for all finitely generated fields $k'$ as above:
$$
\lambda_X^i:=\lim_{\substack{\longrightarrow\\ k'/k_0}} \tilde \lambda_{tr}^i : \CH^i(X)[\ell^\infty]=\lim_{\substack{\longrightarrow\\ k'/k_0}}A^i(X_{k'})[\ell^\infty]\longrightarrow \lim_{\substack{\longrightarrow\\ k'/k_0}} \frac{H^{2i-1}(X_{k'},\Q_\ell/\Z_\ell(i))}{N^{i-1}H^{2i-1}(X_{k'},\Q_\ell(i))}   =\frac{ H^{2i-1}(X ,\Q_\ell/\Z_\ell(i))}{M^{2i-1}(X)}, 
$$
where we use that 
$$
H^{2i-1}(X,\Q_\ell/\Z_\ell(n))\cong \lim_{\substack{\longrightarrow\\ k'/k_0}}H^{2i-1}(X_{k'},\Q_\ell/\Z_\ell(n)) ,
$$
and where 
 $$
M^{2i-1}(X):=\im\left( \lim_{\substack{\longrightarrow \\ k'/k_0}} N^{i-1} H^{2i-1}(X_{k'},\Q_\ell(i))\longrightarrow  H^{2i-1}(X,\Q_\ell/\Z_\ell(i)) \right) .
$$ 
\par If $X$ is smooth and equi-dimensional (e.g. a smooth variety), then Borel--Moore cohomology agrees with ordinary cohomology (see \eqref{eq:ordinary-cohomology}) and the above map yields a cycle map as in (\ref{def:lambda}).

\begin{lemma} \label{lem:M^2i-1}
Assume that $X$ is a smooth projective $k$-variety.
In the following special cases, the group $M^{2i-1}(X)$ can be computed explicitly as follows:
\begin{enumerate}
    \item If $k$ is algebraically closed, then $M^{2i-1}(X)=0$ for all $i$;\label{item:M^2i-1:1}
    \item If $k$ is a finite field, then $M^{2i-1}(X)=0$ for all $i$;\label{item:M^2i-1:2}
    \item \label{item:M^2i-1:3} If $i=1$, $k$ is arbitrary, and $X$ is geometrically integral, then \begin{align*}M^1(X)=\im(H^1(\Spec k,\Q_\ell/\Z_\ell(1))\to H^1(X,\Q_\ell/\Z_\ell(1))).\end{align*}
\end{enumerate}
\end{lemma}
\begin{proof}
Assume first that $k$ is algebraically closed.
By the Weil conjectures proven by Deligne, the group $H^{2i-1}(X,\Q_\ell (i))$ does not contain any nontrivial element that is fixed by the absolute Galois group of a finitely generated subfield $k'\subset k$.
Hence, the natural map
$ 
 H^{2i-1}(X_{k'},\Q_\ell(i))\to  H^{2i-1}(X,\Q_\ell (i))
$ 
is zero for any finitely generated field $k'\subset k$.
This implies $M^{2i-1}(X)=0$ as we want in \eqref{item:M^2i-1:1}.

Assume now that $k$ is a finite field.
The Weil conjectures imply that $H^{2i-1}(X,\Q_\ell/\Z_\ell(i))$ is a finite group (see \cite[Théor\`eme 2]{CTSS}) and so the map $H^{2i-1}(X,\Q_\ell(i))\to H^{2i-1}(X,\Q_\ell/\Z_\ell(i))$ is zero.
This implies $M^{2i-1}(X)=0$ as claimed.

It suffices by a limit argument to prove the last claim in the case where $k$ is an arbitrary finitely generated field.
Let $G=\Gal_k$ be the absolute Galois group of $k$ and let $\bar k$ be an algebraic (or separable) closure of $k$.
Then the Hochschild--Serre spectral sequence from \cite{jannsen} yields an exact sequence
$$
H^1(\Spec k,\Q_\ell(1))\longrightarrow H^1(X,\Q_\ell(1))\longrightarrow H^1(X_{\bar k},\Q_\ell(1))^G,
$$
where we use that $H^0(X_{\bar k},\Q_\ell(1))=\Q_\ell(1)$ because $X$ is geometrically integral.
The Weil conjectures proven by Deligne \cite{deligne} imply that $H^1(X_{\bar k},\Q_\ell(1))^G=0$ and so the first map in the above sequence is surjective.
Moreover, Kummer theory implies that $H^1(\Spec k,\Q_\ell(1))\to H^1(\Spec k,\Q_\ell/\Z_\ell(1)) $ is surjective.
Since $N^0H^1(X,\Q_\ell(1))=H^1(X,\Q_\ell(1))$, we finally conclude
$$
M^1(X)=\im(H^1(X,\Q_\ell(1))\to H^1(X,\Q_\ell/\Z_\ell(1)))=\im (H^1(\Spec k,\Q_\ell/\Z_\ell(1))\to H^1(X,\Q_\ell/\Z_\ell(1))) ,
$$
as we want.
This finishes the proof of the lemma.
\end{proof}

\subsection{Basic properties of \texorpdfstring{$\lambda_X^i$}{lambda}}

\begin{theorem} \label{thm:lambda-body}
Let $X$ be an algebraic scheme over a field $k$ and let $\ell$ be a prime invertible in $k$. 
Then the cycle map
$$
\lambda_X^i  : \CH^i(X)[\ell^\infty] \longrightarrow  \frac{ H^{2i-1}_{BM}(X ,\Q_\ell/\Z_\ell(i))}{M^{2i-1}(X)}  
$$
is injective for $i\leq 2$.
\end{theorem}
\begin{proof}
This follows from Proposition \ref{prop:tilde-lambda}, where we recall our convention that 
$ H^{\ast} (X ,A(n))=H^{\ast}_{BM}(X ,A(n))$ denotes Borel--Moore cohomology.
\end{proof}

The following result generalises \cite[(18.4)]{MS}.

\begin{lemma}\label{lem:image-lambda} Let $X$ be an algebraic scheme over a field $k$ and let $\ell$ be a prime invertible in $k$. Then the image of $\lambda^{i}_{X}$ is given by $$\im(\lambda^{i}_{X})=\frac{N^{i-1} H^{2i-1}_{BM}(X ,\Q_\ell/\Z_\ell(i))}{M^{2i-1}(X)}.$$
\end{lemma}
\begin{proof}This follows from Proposition \ref{prop:image-tilde-lambda} together with the construction of $\lambda_X^i$ via direct limit in Section \ref{subsec:constr-lambda}.
\end{proof}

\begin{corollary} \label{cor:lambda-body}
Let $X$ be an algebraic scheme over a field $k$ and let $\ell$ be a prime invertible in $k$. 
Then $\lambda_X^i$ induces for $i\in \{1,2\}$ isomorphisms
$$
\CH^1(X)[\ell^\infty] \cong  \frac{ H^{1}_{BM}(X ,\Q_\ell/\Z_\ell(1))}{M^{1}(X)}\ \ \text{and}\ \ \CH^2(X)[\ell^\infty] \cong  \frac{N^{1} H^{3}_{BM}(X ,\Q_\ell/\Z_\ell(2))}{M^{3}(X)}    .
$$
\end{corollary}
\begin{proof}
By Lemma \ref{lem:image-lambda},  
$ 
\im(\lambda_X^i)= N^{i-1} H^{2i-1}_{BM}(X ,\Q_\ell/\Z_\ell(i))/ M^{2i-1}(X) .
$ 
Hence the result is an immediate consequence of Theorem \ref{thm:lambda-body}.\end{proof}

Comparing the above construction with Bloch's map from \cite{bloch-compositio}, we get the following:

\begin{lemma} \label{lem:lambda_X^i=Bloch}
If $k$ is algebraically closed and $X$ is a smooth projective variety over $k$, then $M^{2i-1}(X)=0$ for all $i$ and the map
$$
\lambda_X^i: \CH^i(X)[\ell^\infty]\longrightarrow  H^{2i-1}(X ,\Q_\ell/\Z_\ell(i))
$$
agrees with Bloch's map from \cite{bloch-compositio}.
In particular, if $k=\C$ then $\lambda_X^i$ restricted to the subgroup of homologically trivial cycles can be identified with Griffiths Abel--Jacobi map from \cite{griffiths}.
\end{lemma}
\begin{proof}
The vanishing of $M^{2i-1}(X)$ follows from Lemma \ref{lem:M^2i-1}.
Using this it is straightforward to check that our map $\lambda_X^i$ coincides over algebraically closed fields with Bloch's map from \cite{bloch-compositio}.
The comparison with Griffiths' map thus follows from \cite[Proposition 3.7]{bloch-compositio}.
This concludes the proof of the lemma.
\end{proof}


\begin{lemma}\label{lem:reduction-from-k-to-k'}
Let $X$ be an algebraic scheme  over a field $k$ and let $k_0\subset k$ be a finitely generated subfield such that there is a variety $X_0$ over $k_0$ with $X=X_0\times_{k_0} k$. Let $\ell$ be a prime invertible in $k$.
Assume that $\lambda_X^i$ is not injective. 
Let $L\subset k$ be any subfield with $k_0\subset L$.
Then there is a subfield $k'\subset k$ with $L\subset k'$, such that $k'/L$ is finitely generated and 
$$
\lambda_{X'}^i:\CH^i(X')[\ell^\infty]\longrightarrow \frac{ H^{2i-1}(X' ,\Q_\ell/\Z_\ell(i))}{M^{2i-1}(X')}
$$
is not injective, where  $X':=X_{0}\times_{k_0}k'$.
\end{lemma}
\begin{proof}
By assumption there is a  nontrivial element $[z]\in \CH^i(X)[\ell^\infty]$ with $\lambda_X^i([z])=0$.
We can choose a finitely generated extension $k'/L$ such that $z$ is defined over $k'$ and so we get a class 
$$
[z']\in  \CH^i(X') \ \ \text{with}\ \ [z'_k]=[z]\in \CH^i(X)[\ell^\infty]
$$
where $X'=X_{0}\times_{k_0}k'$.
Up to possibly  replacing $k'$ by a larger finitely generated extension of $L$, we can assume that $[z']$ is $\ell^\infty$-torsion and so there is a class
$$
\lambda_{X'}^i([z'])\in \frac{ H^{2i-1}(X' ,\Q_\ell/\Z_\ell(i))}{M^{2i-1}(X')}.
$$ 
Note that $[z'_k]=[z]$ lies in the kernel of $\lambda_X^i$.
Since $\lambda_X^i$ and $\lambda_{X'}^i$ are defined via direct limits over all finitely generated subfields of $k$ and $k'$, respectively, we see that up to possibly replacing $k'$ by a larger finitely generated extension of $L$, we may assume that $\lambda_{X'}^i([z'])=0$ while $[z']\in \CH^i(X')$ is still non-trivial because its base change to $k$ is  nontrivial.
This proves the lemma.
\end{proof}

Next, we have the following.

\begin{lemma} \label{lem:lambda-cl}
Let $X$ be an algebraic $k$-scheme which admits a closed embedding into a smooth equi-dimensional algebraic $k$-scheme (e.g. $X$ is quasi-projective) and let $\ell$ be a prime invertible in $k$. The composition of $\lambda_{X}^i$ with the projection
$$
\frac{ H^{2i-1}(X ,\Q_\ell/\Z_\ell(i))}{M^{2i-1}(X)} \twoheadrightarrow \frac{ H^{2i-1}(X ,\Q_\ell/\Z_\ell(i))}{ H^{2i-1}(X ,\Q_\ell (i))},
$$
followed by the injection
$$
-\delta: \frac{ H^{2i-1}(X ,\Q_\ell/\Z_\ell(i))}{ H^{2i-1}(X ,\Q_\ell (i))} \hookrightarrow H^{2i}(X,\Z_\ell(i))
$$
induced by the Bockstein map $\delta$, 
agrees with the cycle class map $\cl_X^i:\CH^i(X)[\ell^\infty]\longrightarrow H^{2i}(X,\Z_\ell(i))$. 
\end{lemma}
\begin{proof}
This is a direct consequence of Lemma \ref{lem:tilde-lambda-vscl}.
\end{proof}

\begin{lemma} \label{lem:lambda-z1xz2-2}
Let $X_j$ for $j=1,2$ be smooth projective $k$-varieties and let $\ell$ be a prime invertible in $k$.
Let $z_j\in \CH^{i_j}(X_j)$ be classes such that $z_2$ is $\ell^r$-torsion for some positive integer $r$.
Assume that $\cl^{i_1}_{X_1}(z_1)$ is zero modulo $\ell^r$.
Then the $\ell^r$-torsion cycle
$$
z:=z_1\times z_2\in \CH^{i_1+i_2}(X_1\times X_2)
$$
lies in the kernel of the cycle map $ \lambda_{X}^{i_1+i_2}$ from (\ref{def:lambda}).
\end{lemma}
\begin{proof}
Let $k_0$ be a finitely generated subfield over which $X_j$ and the cycle $z_j$ are both defined for $j=1,2$.
Since $z_2$ is $\ell^r$-torsion, we may assume that it is already $\ell^r$-torsion on $X_0$, i.e.\ when viewed as a cycle over $k_0$.
Note moreover that
$$
H^{2i_1}(X_1,\mu_{\ell^r}^{\otimes i_1})=\lim_{\substack{\longrightarrow\\ k'/k_0}}H^{2i_1}({X_1}_{k'},\mu_{\ell^r}^{\otimes i_1}) ,
$$
where $k'$ runs through all finitely generated subfields of $k$ that contain $k_0$ and ${X_1}_{k'}$ denotes the base change to $k'$ of a fixed form of $X_1$ over $k_0$.
We can therefore also assume that the cycle class of $z_1$ over $k_0$ is zero modulo $\ell^r$.
The claim in the lemma follows then from Corollary \ref{cor:lambda-z1xz2}. 
\end{proof}

\section{Schoen's argument over non-closed fields}

The following proposition extends (some version of) the main result from \cite{schoen-product} to non-closed fields.

\begin{proposition} \label{prop:schoen}
Let $k$ be an algebraically closed field and let $C$ be a smooth irreducible curve over $k$.
Let $\ell$ be a prime and let $r$ be a positive integer.  
Let $E$ be an elliptic curve over $k(C)$ whose $j$-invariant is transcendental over $k$.
Up to replacing $C$ by a (possibly ramified) finite cover, the following holds for any $k$-variety $B$, where we denote by $K=k(B\times C)$ the function field of $B\times C$: There is a class $\tau \in \CH_0(E_K)$ of order $\ell^r$ such that for any smooth projective variety $Y$ over $k(B)$, the kernel of the exterior product map
$$
\CH^i(Y)\otimes \Z/\ell^r\longrightarrow \CH^{i+1}(Y_K\times_KE_K)[\ell^r],\ \ \ z \mapsto [z_K\times \tau] 
$$
is contained in the image of
$$
\CH^i(Y)_{\tors }\otimes \Z/\ell^r\longrightarrow \CH^i(Y)\otimes \Z/\ell^r .
$$\end{proposition}
\begin{proof}
Replacing $C$ by the normalization of a projective closure, we may assume that $C$ is smooth and projective.
Note that $k$ is algebraically closed and that the $j$-invariant of $E$ is transcendental over $k$.
Using this, the same argument as in \cite[Lemma 2.7]{schoen-product} shows that up to replacing $C$ by a finite cover, we may  assume that   the curve $E_K$ admits a regular projective model $\mathcal E$ over $C_{k(B)}:=C\times_kk(B)$,  with the following properties:
\begin{enumerate}  [label=(\roman*)]
\item $\mathcal E\to C_{k(B)}$ is a minimal elliptic surface over $k(B)$;\label{item:schoen:1}
\item there is a $k$-rational point on $C$ with induced point $0\in C_{k(B)}$, such that the fibre $F$ of $\mathcal E\to C_{k(B)}$ above $0$ is of type $I_{\ell^rN}$ for some $N\geq 1$;\label{item:schoen:2}
\item if we denote by $F_i$, $i=0,1,\dots , \ell^rN-1$ the components of $F$, then $F_i^2=-2$ for all $i$ and  $F_i\cdot F_{i-1}=1$ for all $i$ if $\ell^rN\neq 2$ and $F_0F_1=2$ if $\ell^rN= 2$, where the index has to be read modulo $ \ell^rN$;\label{item:schoen:3}
\item the model $\mathcal E\to C_{k(B)}$ admits two sections $s_0,s_1$ such that $s_1-s_0$ restricts to a zero-cycle $\tau \in \CH_0(E_K)$ of order $\ell^r$ and such that $s_0$ meets $F_0$, while the (unique) component of $F$ that meets $s_1$ is of the form $F_{mN}$ with $m$ coprime to $\ell$.\label{item:schoen:4}
\end{enumerate}
The same argument as in \cite[Lemma 2.8]{schoen-product} then shows that there is a divisor $D$ on $\mathcal E$ which is supported on the special fibre $F$ and such that the following holds:
\begin{enumerate}[label=(\roman*)]
\setcounter{enumi}{4}
\item if $D'$ is another divisor on $\mathcal E$ that is supported on some fibres of $\mathcal E\to C_{k(B)}$, then $D'\cdot D\equiv 0 \mod \ell^r N$;\label{item:schoen:5}
\item  $(s_1-s_0)\cdot D=mN \cdot \chi$ for some integer $m$ that is coprime to $\ell^r$ and some zero-cycle $\chi\in \CH_0(F)$ of degree 1.\label{item:schoen:6}
\end{enumerate}

To conclude the argument, we consider the model
$$
\mathcal X:=Y\times_{k(B)}\mathcal E=(Y\times_{k(B)} C_{k(B)}) \times_{C_{k(B)}}\mathcal E\longrightarrow  C_{k(B)}
$$
with special fibre $X_0=Y\times_{k(B)} F$ and generic fibre $X_\eta=Y_K\times_KE_K$.
Since $Y$ is smooth over $k(B)$ and $\mathcal E$ is regular, the model $\mathcal X$ is regular as well.
Let $z\in \CH^i(Y)$. 
Since $s_1-s_0$ is a divisor class on $\mathcal E$, we can consider the  exterior product cycle
$$
\xi:=z\times (s_1-s_0)\in \CH^{i+1}(\mathcal X).
$$
The restriction  $\xi_\eta\in \CH^{i+1}(Y_K\times_KE_K)$ of this cycle to the generic fibre of $\mathcal X\to C_{k(B)}$ agrees with 
$$
[z_K\times \tau]\in \CH^{i+1}( Y_K\times_KE_K),
$$
where 
$ z_K\in \CH^i(Y_K)$ denotes the base change of $z$ and $\tau\in \CH_0(E_K)$ is the class of order $\ell^r$ from \ref{item:schoen:4}.

To prove the proposition, we assume that $\xi_\eta=0$ and we then aim to show that $z$ is the sum of a torsion class and a class that is $\ell^r$-divisible.
To this end, note that the localization formula together with the assumption $\xi_\eta=0$ implies that there are closed points $c_1,\dots ,c_n\in C_{k(B)}$ with $c_j\neq 0$ for all $j$ such that
$$
\xi\in \im\left(\CH^{i}(X_0)\oplus \bigoplus_j \CH^i(X_{c_j}) \longrightarrow \CH^{i+1}(\mathcal X)\right) .
$$
Hence,
$$
\xi=\xi'+\xi''\ \ \text{where}\ \ \xi'\in  \im\left(\CH^{i}(X_0) \longrightarrow \CH^{i+1}(\mathcal X)\right)\ \ \text{and}\ \  \xi''\in \im\left(  \bigoplus_j \CH^i(X_{c_j}) \longrightarrow \CH^{i+1}(\mathcal X)\right) .
$$
Note that $X_0=Y\times_{k(B)} F$. 
Since $F$ is a cycle of smooth rational curves, the natural map
$$
(\CH^{i}(Y)\otimes \CH^0(F))\oplus (\CH^{i-1}(Y)\otimes \CH^1(F))\longrightarrow \CH^{i}(X_0)
$$
is surjective.
It follows that we can write
$$
\xi'=\sum_j \xi_{1j}\times D_j'+\xi_2\times \zeta
$$
for some $\xi_{1j}\in \CH^{i}(Y)$, $D'_j\in \CH^0(F)$,  $\xi_2\in \CH^{i-1}(Y)$, and $\zeta\in \CH^1(F)$.
Let $Y\times D$ denote the pullback of the divisor $D$ via the natural projection $\pr_2:\mathcal X\to \mathcal E$.
Since $\mathcal X$ is regular, we may consider the intersection product $\xi\cdot (Y\times  D) \in \CH^{i+2}(\mathcal X)$.
Since $D$ is supported on the special fibre $F$,  the support of $\xi''$ is disjoint from the support of $Y\times D$ and so $\xi''\cdot  (Y\times  D)=0$.
Hence, 
$$
\xi\cdot (Y\times  D)=(\sum_j \xi_{1j}\times D_j'+\xi_2\times \zeta)\cdot (Y\times D) \in \CH^{i+2}(\mathcal X).
$$
By item \ref{item:schoen:5}, $(\xi_{1j}\times D'_j)\cdot (Y\times D)\equiv 0\mod \ell^rN $ for all $j$ and so
$$
\xi\cdot (Y\times  D)= \xi_2\times (\zeta\cdot D) =0 \in \CH^{i+2}(\mathcal X)/\ell^rN ,
$$
because $\zeta\cdot D=0$ for degree reasons. 
Hence,
$$
\xi\cdot (Y\times  D)\in \ell^rN\cdot \CH^{i+2}(\mathcal X).
$$
On the other hand, $\xi= z \times(s_1-s_0)$ and so item \ref{item:schoen:6} implies
$$
\xi\cdot (Y\times  D)=z \times mN \chi .
$$
Consider the natural proper map
$$
\pi:\mathcal X=Y\times_{k(B)}\mathcal E \longrightarrow Y
$$
of $k(B)$-varieties.
Then the above computations show that
$$
\pi_\ast (\xi\cdot (Y\times  D))=mN\cdot z =\ell^rN\cdot z' \in \CH^i(Y)
$$
for some $z'\in \CH^i(Y)$.
Hence,
\begin{align}
\label{eq:t}
t:=m \cdot z-\ell^r \cdot z'  \in \CH^i(Y) 
\end{align}
is $N$-torsion.
Since $m$ is coprime to $\ell$, we find that the class of $z$ in $\CH^i(Y)/\ell^r$ may be represented by a torsion class, as we want.
This concludes the proof of the proposition.
\end{proof}

For later use we record here the following variant of the above argument.

\begin{proposition} \label{prop:schoen-2}
Let $k$ be an arbitrary field.
Let $Y$ be a smooth projective $k$-variety. 
Then there is an elliptic curve $E$ over $K=k(\CP^1)$ (in fact the Legendre elliptic curve $y^2=x(x-1)(x-t)$) and a class $\tau\in\CH_{0}(E)$ of order $2$ such that the exterior product map
$$
\CH^{i}(Y)\otimes\Z/2\longrightarrow\CH^{i+1}(Y_K\times_KE)[2],\ \  z \mapsto [z_K\times \tau]
$$ is injective.
\end{proposition}
\begin{proof}
We let $E$ be the generic fibre of the Legendre family $y^2=x(x-1)(x-t)$, where $t$ denotes an affine coordinate on $\CP^1$.
Let $\mathcal E\to \CP^1$ be a minimal elliptic surface whose generic fibre is $E$.
The special fibre $F$ above $t=0$ is then of type $I_2$: $F=F_0\cup F_1$ with $F_j^2=-2$ and $F_0\cdot F_1=2$.
Note further that each $2$-torsion point of $E$ is $K$-rational (see \cite{Igusa}) and it extends to a section of $\mathcal E\to \CP^1$.
It follows in particular that there are sections $s_0,s_1$ of $\mathcal E\to \CP^1$ such that $s_j$ meets $F_j$ and such that $s_1-s_0$ restricts to a zero-cycle $\tau\in \CH_0(E)$ of order $2$.
Let $z\in \CH^i(Y)$ be nonzero modulo $2$.
Let $z_K\in \CH^i(Y_K)$ denote the base change of $z$.
To prove the proposition, it is then enough to show that the $2$-torsion class
$$
z_K\times \tau \in \CH^{i+1}(Y_K\times_KE)
$$
is nonzero.
This follows by exactly the same argument as in the proof of Proposition \ref{prop:schoen} above.
The main difference is that in the current situation, the special fibre $F$ is of type $I_2$ and so the integer $N$ in the proof of Proposition \ref{prop:schoen} equals $1$, so that the class $t$ in (\ref{eq:t}) is not only torsion, but in fact zero.
\end{proof}

\begin{remark} Let $C$ be a smooth irreducible projective curve over a field $k$ and let $\ell$ be a prime. The proof of Proposition \ref{prop:schoen-2} shows more generally that if $E$ is an elliptic curve over $K=k(C)$ whose $\ell$-torsion points are all $K$-rational and in addition there is a $K$-rational point $0\in C$ so that the fibre $F$ of the minimal model $\mathcal E\to C$ above $t=0$ is of type $I_{\ell}$, then there is a class $\tau\in\CH_{0}(E)$ of order $\ell$ such that for any smooth projective $k$-variety $Y$ the exterior product map
$$
\CH^{i}(Y)\otimes\Z/\ell\longrightarrow\CH^{i+1}(Y_K\times_KE)[\ell],\ \  z \mapsto [z_K\times \tau]
$$ is injective.\par A concrete example for the prime $\ell=3$ is given by the elliptic curve $E$ over $\mathbb{Q}(\omega)(\CP^{1})$ $(\omega^{3}=1,\omega\neq 1)$ defined by the degree $3$ equation $$x^{3}_{0}+x^{3}_{1}+x^{3}_{2}=3\lambda x_{0}x_{1}x_{2}$$ in $\CP^{2}$; see \cite{Igusa}. \end{remark}

\section{A construction of Koll\'ar, Hassett--Tschinkel, and Totaro}

In this section we use some arguments of Koll\'ar \cite{kollar92}, Hassett--Tschinkel (cf.\ \cite[Introduction]{totaro-IHC-2}) 
and Totaro \cite{totaro-IHC-2} to prove the following:

\begin{proposition}\label{prop:kollar}
Let $k$ be an algebraically closed field and let $\ell$ be a prime invertible in $k$.
If the characteristic of $k$ is positive, assume that $k$ has positive transcendence degree over its prime field.
There is smooth projective hypersurface $S\subset \CP^3_{k(\CP^1)}$ such that the cycle class map
$$
\CH_0(S)/\ell\longrightarrow H^4(S,\mu_{\ell}^{\otimes 2})
$$
is not injective up to torsion.
That is, there is a zero-cycle $z\in \CH_0(S) $ with trivial cycle class in $H^4(S,\mu_{\ell}^{\otimes 2})$ such that $z\notin \CH_0(S)_{\tors}+\ell\cdot \CH_0(S)$.
\end{proposition}

\begin{proof}
Our examples arise due to the failure of the integral Hodge and Tate conjectures for certain hypersurfaces in $\CP_{k}^{1}\times_{k}\CP_{k}^{3}$ of bidegree $(3,\ell^{2})$. 
To this end, let $\lambda\in k$ be transcendental over the prime field if $k$ has positive characteristic, and let $\lambda:=\ell$ if $k$ has characteristic zero.
We then consider the smooth hypersurface $\mathcal{X}\subset \CP^{1}_{u,t}\times\CP^{3}_{x_{0},x_{1},x_{2},x_{3}}$ over $k$, defined by the equation
$$
u^{3}x^{\ell^{2}}_{0}+tu^{2}x^{\ell^{2}}_{1}+t^{2}ux^{\ell^{2}}_{2}+t^{3}x^{\ell^{2}}_{3}+\lambda(u^{3}x^{\ell^{2}}_{3}-t^{3}x^{\ell^{2}}_{0}+u^{3}x^{\ell^{2}}_{2}-t^{3}x^{\ell^{2}}_{1})=0 .
$$ 

Let $S:=\mathcal{X}_{\eta}$ be the generic fibre of the first projection $\pr_{1}\colon \mathcal{X}\to\CP_{k}^{1}$. We aim to show that the cycle class map $\cl^{2}_{S}\colon\CH_0(S)/\ell\longrightarrow H^4(S,\mu_{\ell}^{\otimes 2})$ is not injective up to torsion. 

\setcounter{step}{0}
\begin{step}\label{st:step1} Let $n:=\gcd\{\deg(z)|z\in\CH_{0}(S)\}$ be the index of $S$. 
If $\mathcal{Z}^{4}(\mathcal{X})$ denotes the cokernel of the cycle class map $\cl^{2}_{\mathcal{X}}\colon\CH^{2}(\mathcal{X})_{\Z_\ell}\to H^{4}(\mathcal {X},\Z_{\ell}(2))$, then $\mathcal{Z}^{4}(\mathcal{X})\cong\Z_{\ell}/n\Z_{\ell}$.

\begin{proof} 
We have the following commutative diagram \begin{center} \begin{tikzcd}
\CH^{3}(\CP^{1}\times_{k}\CP^{3})_{\Z_\ell} \arrow[r, "\cl^{3}"]                         & {H^{6}(\CP^{1}\times_{k}\CP^{3},\Z_{\ell}(3))}           \\
\CH^{2}(\mathcal{X})_{\Z_\ell} \arrow[r, "\cl^{2}_{\mathcal{X}}"] \arrow[u, "i_{*}"] & {H^{4}(\mathcal{X},\Z_{\ell}(2))}, \arrow[u, "i_{*}"]
\end{tikzcd}\end{center} where $i_{*}\colon H^{4}(\mathcal{X},\Z_{\ell}(2))\to H^{6}(\CP^{1}\times_{k}\CP^{3},\Z_{\ell}(3))$ is an isomorphism by the Weak Lefschetz theorem \cite[Theorem VI.7.1]{milne}. 
Recall that $H^{6}(\CP^{1}\times_{k}\CP^{3},\Z_{\ell}(3))$ is the free $\Z_{\ell}$-module of rank $2$ generated by $\ell_{1}:=\cl^{3}(\CP_{k}^{1}\times pt)$ and $\ell_{2}:=\cl^{3}(pt\times L)$, where $L\subset\CP^{3}_{k}$ is any one-dimensional linear subspace. 
Let $L\subset\CP^{3}_{k}$ be the subspace defined by the equations $x_{0}=0$ and $x_{3}=\zeta x_{2}$, where $\zeta\in k$ with $\zeta^{\ell^{2}}=-1$. 
An easy check shows $(1:0)\times L\subset\mathcal{X}$. The $1$-cycle $\tilde{\ell_{2}}:=[(1:0)\times L]\in\CH^{2}(\mathcal{X})$ clearly then satisfies $i_{*}\cl^{2}_{\mathcal{X}}(\tilde{\ell_{2}})=\ell_{2}$. Pick $z\in\CH^{2}(\mathcal{X})$ such that $\deg(z|_{S})=n$. 
By replacing $z$ with $z':=z-m\tilde{\ell_{2}}$, where the integer $m$ is determined by ${\pr_{2}}_{*}(z)=m[L]$ with $\pr_2:\mathcal X\to \CP^3$, we may assume that ${\pr_{2}}_{*}(z)=0$. It follows that $i_{*}\cl^{2}_{\mathcal{X}}(z)=n\ell_{1}$. 
To conclude we notice that the image of the composite 
$$
\CH^{2}(\mathcal{X})_{\Z_{\ell}}\longrightarrow H^{4}(\mathcal{X},\Z_{\ell}(2))\overset{{\pr_{1}}_{*}}\longrightarrow  H^{0}(\CP_{k}^{1},\Z_{\ell}(0))=\Z_{\ell}
$$ 
is $n\Z_{\ell}$, where $n$ is the index of $S$.
\end{proof}
\end{step}

\begin{step}\label{st:step2}
We have $n\in\ell\Z_{\ell}$.
\end{step}

\begin{proof} 
The proof is essentially the same as in \cite[Theorem 2.1]{totaro-IHC-2}, where the case $\ell=2$ and $k=\bar{\mathbb{Q}}$ is treated. 
We briefly recall the argument for the readers convenience.

Note that $\mathcal{X}$ degenerates via $\lambda\to 0$ to the hypersurface 
$$
u^{3}x^{\ell^{2}}_{0}+tu^{2}x^{\ell^{2}}_{1}+t^{2}ux^{\ell^{2}}_{2}+t^{3}x^{\ell^{2}}_{3}=0,
$$ 
in $\CP^{1}\times\CP^{3}$ over some algebraically closed field extension $F/\bar{\mathbb{F}}_{q}$, where $q=\ell$ if $\Char(k)=0$ and $q=\Char(k),$ otherwise.
A straightforward specialization argument enables us to reduce our task to proving that the hypersurface 
$$
Y:=\{x^{\ell^{2}}_{0}+tx^{\ell^{2}}_{1}+t^{2}x^{\ell^{2}}_{2}+t^{3}x^{\ell^{2}}_{3}=0\}
$$ 
in $\CP^{3}$ over $F((t))$ has no rational point over any extension $F((s))$ of $F((t))$ whose degree is not divisible by $\ell$. 
For a contradiction, we assume that there is an extension $F((s))/F((t))$ of degree $d$ with $\gcd(d,\ell)=1$ and Laurent series $t(s),x_{i}(s)\in F((s))$ satisfying the equation of $Y$. 
Then the valuations of the $4$ terms in the equation, given by $\ord_{s}(t^{i}x^{\ell^{2}}_{i})\equiv id \mod \ell^{2}$, are all different, forcing all the $x_{i}(s)$ to be zero. 
This certainly cannot be a point of the projective space, contradicting our hypothesis.
\end{proof} 

Finally we are in the position to conclude the proof of Proposition \ref{prop:kollar}. 
By Step \ref{st:step1} we can pick a $1$-cycle $\alpha\in\CH^{2}(\mathcal{X})$ such that $\cl^{2}_{\mathcal{X}}(\alpha)=n\ell_{1}$. 
We set $z:=\alpha|_{S}\in\CH^{2}(S)$ and note that $\cl^{2}_{S}(z)=0\in H^{4}(S,\mu_{\ell}^{\otimes 2})$ because $n\in\ell\Z_{\ell}$ by Step \ref{st:step2}. 
We claim that $z\notin\CH^{2}(S)_{\tors}+\ell\CH^{2}(S)$. 
Indeed, if we write $z=z_{1}+\ell z_{2}$ for some $z_{1}\in\CH^{2}(S)_{\tors}$ and $z_{2}\in \CH^{2}(S)$, then $n=\deg(z)=\ell\deg(z_{2})$. 
This yields a contradiction, as the integer $n$ is the index of $S$. 
\end{proof}

\section{Proof of the main results}

\begin{proof}[Proof of Theorem \ref{thm:Sch}]The injectivity of $\lambda_X^i$ for $i=1,2$ follows from Theorem \ref{thm:lambda-body}.
The isomorphisms in question then follow from the description of the image of $\lambda_X^i$ in Corollary \ref{cor:lambda-body}.
In the case $i=1$, we note in addition that $N^0H^i(X,A(n))=H^i(X,A(n))$.
\end{proof}

\begin{proof}[Proof of Theorem \ref{thm:main}]
Let $\ell$ be a prime and let $k$ be a field of characteristic different from $\ell$.
We aim to construct a finitely generated field extension $K/k$ and a smooth projective threefold $X$ over $K$ such that
$$
\lambda_X^3:\CH^3(X)[\ell^\infty]\longrightarrow H^5(X,\Q_\ell/\Z_\ell(2))/M^5(X)
$$
is not injective.
To prove this, we are free to replace $k$ by a finitely generated field extension and so we can assume that $k$ has positive transcendence degree over its prime field.
Lemma \ref{lem:reduction-from-k-to-k'} allows us to reduce further to the case where $k$ is algebraically closed and of positive transcendence degree over its prime field.
It then follows from Proposition \ref{prop:kollar} that there is a smooth projective hypersurface $S\subset \CP^3_{k(\CP^1)}$ 
which carries a zero-cycle $z_1\in \CH_0(S)$ with trivial cycle class in $ H^4(S,\mu_{\ell}^{\otimes 2})$ such that 
\begin{align} \label{eq:z1-prop-kollar}
    z_1\notin \CH_0(S)_{\tors}+\ell\cdot \CH_0(S) .
\end{align}
Let $C$ be any smooth curve over $k$ and let $E$ be an elliptic curve over $k(C)$ whose $j$-invariant is transcendental over $k$ (e.g.\ we could take $C=\CP^1$ and let $E$ be the Legendre elliptic curve). 
We then apply Proposition \ref{prop:schoen} to $C$, $E$, $B:=\CP^1$, and to the $k(B)$-variety $Y:=S$.
We let $K:=k(B\times C)$ and find that up to replacing $C$ by a finite cover, there is a zero-cycle $\tau\in \CH^1(E_K)$ of order $\ell$ such that
the kernel of the exterior product map
\begin{align} \label{eq:exterior-product}
\CH^2(S)\otimes \Z/\ell \longrightarrow \CH^{3}(S_K\times_K E_K)[\ell ],\ \ \ z \mapsto [z_K\times \tau ] 
\end{align}
is contained in the image of
$$
\CH^2(S)_{\tors }\otimes \Z/\ell\longrightarrow \CH^2(S)\otimes \Z/\ell .
$$

We finally set $X:= S_{K}\times_{K}E_{K}$ and consider the zero-cycle 
$$
z:=[{z_1}_K\times \tau]\in \CH^3(X).
$$
Since $\tau$ has order $\ell$, the zero-cycle $z$ is $\ell$-torsion.
By (\ref{eq:z1-prop-kollar}) and the above description of the kernel of the exterior product map in (\ref{eq:exterior-product}), we find that $z$ is nontrivial.
On the other hand, since the cycle class of $z_1$ in $ H^4(S,\mu_{\ell}^{\otimes 2})$ vanishes while $\tau$ is $\ell$-torsion, Lemma \ref{lem:lambda-z1xz2-2} implies that
$$
\lambda_X^3(z)=0\in H^5(X,\Q_\ell/\Z_\ell(2))/M^5(X).
$$
This concludes the proof of the theorem.
\end{proof}

\begin{proof}[Proof of Corollary \ref{cor:lambda_X^i-injective}]
One direction is Theorem \ref{thm:Sch}.
The other direction follows from Theorem \ref{thm:main} by taking products with projective spaces.
\end{proof}

\begin{proof}[Proof of Corollary \ref{cor:integral-Jannsen}]
This follows from Lemma \ref{lem:lambda-cl} and Theorem \ref{thm:main}.
\end{proof}


\begin{proof}[Proof of Theorem \ref{thm:main-2}]
By \cite[Theorem 8.5]{parimala-suresh}, there is a smooth projective surface $Y$ (given as a conic fibration  over a hyperelliptic curve) over a field $k$ such that $Y(k)\neq \emptyset$ and such that the cycle class map $\CH^2(Y)/2\to H^4(Y,\mu_{2}^{\otimes 2})$ is not injective.
In \textit{loc.\ cit.}, the example is defined over $\Q_3$; a straightforward limit argument then allows us to assume that $k\subset \Q_3$ is a finitely generated subfield. 
We pick a zero-cycle $z\in\CH^{2}(Y)$ that is nonzero modulo $2$ but its cycle class $\cl^{2}_{Y}(z)$ is zero modulo $2$. 
Let now $K=k(\CP^1)$. 
Then Proposition \ref{prop:schoen-2} implies that there is an elliptic curve $E$ over $K$ (in fact the Legendre elliptic curve) and a class $\tau\in\CH_{0}(E)$ of order $2$ such that the cycle $z_{K}\times\tau\in\CH^{3}(X)[2]$ is nonzero, where $X$ is the smooth projective threefold $Y_K\times_K E$. It follows from Lemma \ref{lem:lambda-z1xz2-2} that $\lambda_X^3(z_{K}\times\tau)=0$, proving the non injectivity of $\lambda_X^3$ for the prime $\ell=2$.
Since $E$ is an elliptic curve, it has a rational point and since $Y$ has a rational point as well, we find that $X(K)\neq \emptyset$, as we want.
\end{proof}

 \section*{Acknowledgements} 
 We thank the referee for his or her useful comments. This project has received funding from the European Research Council (ERC) under the European Union's Horizon 2020 research and innovation programme under grant agreement No 948066 (ERC-StG RationAlgic).\\
 \par \textbf{Conflicts of interest:} none.


\end{document}